\documentclass[twoside,12pt]{article}
\usepackage[utf8]{inputenc}
\usepackage{mathrsfs}
\usepackage{amssymb,amsmath,amsthm}
\usepackage{graphicx}
\usepackage{color}
\usepackage{tikz}
\usepackage[top=2cm,bottom=2cm,left=1.8cm,right=1.8cm]{geometry}
\usepackage{float,caption,subcaption}
\usepackage{diagbox}
\usepackage{algorithm}
\usepackage{algpseudocode}

\usepackage[colorlinks,linkcolor=blue,anchorcolor=blue,citecolor=blue]{hyperref}
\input{epsf}

\newcommand{\bd}{\begin{description}}
\newcommand{\ed}{\end{description}}
\newcommand{\bi}{\begin{itemize}}
\newcommand{\ei}{\end{itemize}}
\newcommand{\be}{\begin{enumerate}}
\newcommand{\ee}{\end{enumerate}}
\newcommand{\beq}{\begin{equation}}
\newcommand{\eeq}{\end{equation}}
\newcommand{\beqs}{\begin{eqnarray*}}
\newcommand{\eeqs}{\end{eqnarray*}}

\definecolor{DarkGreen}{rgb}{0.2,0.6,0.3}

\usepackage{multicol}
\usepackage{multirow}
\usepackage{makecell}

\newtheorem{theorem}{Theorem}[section]

\newtheorem{lemma}{Lemma}[section]
\newtheorem{definition}{Definition}[section]
\newtheorem{corollary}[theorem]{Corollary}

\newtheorem{proposition}{Proposition}[section]

\newtheorem{problem}{Problem}

\begin{document}
\title{\textbf{Multiplicity for Partially Ordered Sets\footnote{The work was supported by the Hungarian National Research, Development and Innovation Office NKFIH (No.~SSN135643 and K132696); the National Science Foundation of China (Nos. 12471329 and 12061059).}}}

\author{Gyula O.H. Katona\thanks{HUN-REN Alfr\'ed R\'enyi Institute of Mathematics, Budapest Reaaltanoda utca 13--15, 1053, Hungary. {\tt katona.gyula.oh@renyi.hu}},\ \ \ Yaping Mao\footnote{Corresponding author: Academy of Plateau Science and Sustainability, and School of Mathematics and Statistics, Qinghai Normal University, Xining, Qinghai 810008, China. {\tt yapingmao@outlook.com; myp@qhnu.edu.cn}}}
\date{}
\maketitle

\begin{abstract}
Let $\mathcal Q=\{Q_a:a\geq1\}$ be a nested family of finite posets such that $Q_a\subseteq Q_{a+1}$ and $|Q_a|<|Q_{a+1}|$. For a poset $Q$, let $\mathcal C_t(Q)$ denote the set of all strict $t$-chains in $Q$. Given an $r$-coloring of $\mathcal C_t(Q_a)$ and posets $P_1,\ldots,P_r$, a weak copy of $P_i$ is called monochromatic of color $i$ if all $t$-chains in the copy have color $i$; the strong version is defined in the same way for induced copies. The corresponding weak and strong multiplicity parameters are the minimum possible total number of such monochromatic copies in the host poset.
For the Boolean lattice $B_n$, define
$E_n={(S,T,U)\in B_n^3:S\subsetneq T\subsetneq U,\ |S|+|T|=|U|}.$
For a two-coloring $\chi:B_n\to{0,1}$, a triple $(S,T,U)\in E_n$ is monochromatic if
$\chi(S)=\chi(T)=\chi(U)$. Let $R^{\mathrm{arith}}_2$ be the least integer $n$ such that every two-coloring of $B_n$ contains a monochromatic triple in $E_n$, and let $M^{\mathrm{arith}}_2(B_n)$ be the minimum number of monochromatic triples in $E_n$ over all two-colorings of $B_n$. We prove that
$R^{\mathrm{arith}}_2=9.$
Moreover,
$|E_n|=\binom{2n}{n}-[x^n](1+x+x^2)^n-2^n+1
=\frac{4^n}{\sqrt{\pi n}}\bigl(1+o(1)\bigr),$
and
$
2^{\delta n+o(n)}
\le M^{\mathrm{arith}}_2(B_n)
\le 2^{\gamma n+o(n)},
$
where $\delta\approx 1.356779$ and $\gamma\approx 1.567837$ are explicit entropy constants.
For general nested host families, we prove a double-counting lower bound for strong poset multiplicity. For an arbitrary finite host poset $R$, we also introduce a Fourier--M\"obius method and give an exact Fourier expansion for strong multiplicity, a Parseval-type error bound, and a spectral lower bound.\\[0.2cm]

\end{abstract}

\section{Introduction}

Ramsey theory studies the unavoidable appearance of ordered substructures under arbitrary colorings. Since the foundational work of Ramsey, this point of view has become central in algebra, geometry, logic, ergodic theory, poset theory, number theory, set theory, finite fields, and related areas; see \cite{GRS90}. For applications to communications, information theory, game theory, and coding theory, we refer the reader to the survey of Rosta \cite{Rosta04}.

A \emph{partially ordered set}, or \emph{poset}, is a pair $(X,\leq)$ where $X$ is a set and $\leq$ is a reflexive, anti-symmetric, and transitive relation. A pair $x,y\in X$ is \emph{comparable} if $x\leq y$ or $y\leq x$, and a $k$-chain is a set of $k$ distinct pairwise comparable elements. If $P$ and $Q$ are posets, an injection $f:P\rightarrow Q$ is a \emph{weak embedding} if $f(x)\leq f(y)$ whenever $x\leq y$ in $P$; the image is called a weak copy of $P$ in $Q$. We say that $Q$ is $P$-free if there is no weak embedding of $P$ into $Q$. An injection $f:P\rightarrow Q$ is a \emph{strong embedding} if $f(x)\leq f(y)$ if and only if $x\leq y$ in $P$; in this case the image is an \emph{induced copy} of $P$ in $Q$. In the weak setting below, copies are counted as embeddings unless explicitly stated otherwise.

The \emph{height} $h(P)$ of a poset $P$ is the number of elements in a longest strictly increasing sequence $x_1<x_2<\cdots<x_h$. The \emph{width} $w(P)$ of $P$ is the number of elements in a largest set of pairwise incomparable elements. If $X$ is an $n$-element set, then the \emph{Boolean lattice of dimension $n$}, denoted by $B_n$, is the power set $2^X$ equipped with the inclusion order $\subseteq$.

For a poset $Q$ and an integer $t\geq1$, let $\mathcal C_t(Q)$ denote the set of all strict $t$-chains in $Q$. An $r$-coloring of $\mathcal C_t(Q)$ is a map $\chi:\mathcal C_t(Q)\rightarrow[r]=\{1,2,\ldots,r\}$. Such a coloring is \emph{exact}, or surjective, if every color in $[r]$ is used. For a copy $P'\subseteq Q$ and a coloring $\chi$ of $\mathcal C_t(Q)$, we say that $P'$ is \emph{monochromatic of color $i$} if every $t$-chain of $P'$ has color $i$.

Ramsey theory on posets was initiated by Ne\v{s}et\v{r}il and R\"odl \cite{NR84}, with an emphasis on induced copies. Further developments appear in \cite{DKT91,LT22,McColm91,TrotterRamsey}. In recent years, much attention has been devoted to Boolean Ramsey problems, where the host poset is a Boolean lattice and one seeks monochromatic induced copies of given posets \cite{AW17,CS18,GRS99,JLM15}.

\subsection{Poset Ramsey numbers}

Let $\mathcal Q=\{Q_a:a\geq1\}$ be a family of posets such that $Q_a$ is a subposet of $Q_{a+1}$ and $|Q_a|<|Q_{a+1}|$ for every $a$. The study of chain colorings in posets goes back at least to the generalized Lubell inequality of Griggs, Stahl, and Trotter \cite{GST84}. Cox and Stolee \cite{CS18} introduced a poset Ramsey framework using pographs and weak embeddings, with Boolean lattices as host posets. Katona et al. \cite{KMOW25} gave a formulation for general poset families that treats both weak and strong embeddings.

\begin{definition}
For given posets $P_1,P_2,\ldots,P_r$, the \emph{weak poset Ramsey number}
\[
\operatorname{R}_{r,t}(\mathcal Q\,|\,P_1,P_2,\ldots,P_r)
\]
is the smallest number $n$ such that every $r$-coloring of $\mathcal C_t(Q_n)$ contains, for some $1\leq i\leq r$, a weak embedding of $P_i$ that is monochromatic of color $i$.
\end{definition}

We write $\operatorname{R}_{r,t}(\mathcal Q\,|\,P)$ for $\operatorname{R}_{r,t}(\mathcal Q\,|\,P_1,P_2,\ldots,P_r)$ when $P_1=\cdots=P_r=P$. If $t=1$, we write $\operatorname{R}_{r}(\mathcal Q\,|\,P_1,P_2,\ldots,P_r)$; if $t=1$ and $r=2$, we write $\operatorname{R}(\mathcal Q\,|\,P_1,P_2)$; and if $r=1$ and $t=1$, we write $\operatorname{R}(\mathcal Q\,|\,P)$.

The \emph{strong poset Ramsey number} $\operatorname{R}_{r,t}^{\sharp}(\mathcal Q\,|\,P_1,P_2,\ldots,P_r)$ is defined analogously, with strong embeddings in place of weak embeddings.

\begin{itemize}
\item If $\mathcal Q=\mathcal C=\{C_n:n\geq1\}$ is a nested family of chains, then $\operatorname{R}_{r,t}(\mathcal C\,|\,P_1,P_2,\ldots,P_r)$ is the \emph{weak chain Ramsey number}.
\item If $\mathcal Q=\mathcal B=\{B_n:n\geq1\}$ is the family of Boolean lattices, then $\operatorname{R}_{r,t}(\mathcal B\,|\,P_1,P_2,\ldots,P_r)$ is the \emph{weak Boolean Ramsey number}. In this case we also write $\operatorname{R}_{r,t}(P_1,P_2,\ldots,P_r)$ for short.
\end{itemize}

The corresponding strong chain Ramsey number and strong Boolean Ramsey number are defined in the same way. If $t=1$ and $r=2$, then we write $\operatorname{R}^{\sharp}(\mathcal B\,|\,P_1,P_2)$ for the strong Boolean Ramsey number.

The Boolean Ramsey number $\operatorname{R}(\mathcal B\,|\,B_s,B_t)$ has been studied by Axenovich and Walzer \cite{AW17}, Lu and Thompson \cite{LT22}, Cox and Stolee \cite{CS18}, Gr\'osz, Methuku, and Tompkins \cite{GMT17,GMT23}, Bohman and Peng \cite{BP2021}, and Walzer \cite{Walzer15}. Axenovich and Winter \cite{AW23} studied $\operatorname{R}(P,B_t)$ for a fixed poset $P$ and a growing Boolean lattice $B_t$. Further off-diagonal results were obtained by Winter for complete multipartite posets \cite{Winter23} and chain compositions \cite{WinterIII}, and by Axenovich and Winter for an $N$-shaped poset \cite{WinterII}. Rainbow generalizations of Boolean Ramsey numbers were investigated in \cite{CGLMNPV22,CCLL20,KMOWY26,Patkos20,Patkos26}.

\subsection{Multiplicity problems}

Goodman \cite{Goodman59} initiated the graph Ramsey multiplicity problem in 1959. For a graph $G$ and an integer $n$, the multiplicity $M(G;n)$ is the minimum number of monochromatic copies of $G$ in a $2$-edge-coloring of $K_n$. For graphs $G_1,\dots,G_k$, Ramsey multiplicity asks for the minimum total number of monochromatic copies of $G_1,\dots,G_k$ in a $k$-edge-coloring at the corresponding Ramsey threshold. This problem has been widely studied; see the survey \cite{BR80} and the papers \cite{CFSW22,CFSW23,Fox08,JST96}.

Arithmetic Ramsey multiplicity problems form another major source of motivation. Graham, R\"odl, and Ruci\'nski \cite{GRR96} investigated the minimum number of monochromatic solutions to $x+y=z$ in $2$-colorings of $[1,n]$. Let $\operatorname{M}(a,n)$ be the corresponding minimum for $x+y=az$. For $a=1$, one has $\operatorname{M}(1,n)=\frac{n^2}{22}(1+o(1))$ \cite{RZ98,S99,D03}. Parrilo et al. \cite{PRS08} obtained asymptotic bounds for $\operatorname{M}(2,n)$. Ramsey multiplicity on integers and graphs has connections to games \cite{Harary82}, abelian groups \cite{SW16}, and vector spaces \cite{RS23}, and it has been studied using methods such as graph limits and flag algebras \cite{MN23,PPSS22}.

We now define the multiplicity parameters used in this paper. Throughout these definitions, we assume that each poset under consideration contains at least one $t$-chain, so that monochromaticity is not vacuous.

\begin{definition}
Let $\mathcal Q=\{Q_a:a\geq1\}$ be a family of finite posets with $Q_a\subseteq Q_{a+1}$ and $|Q_a|<|Q_{a+1}|$ for all $a\geq1$. Let $P_1,\dots,P_r$ be finite posets, and let $n_0=\operatorname R_{r,t}(\mathcal Q\mid P_1,\dots,P_r)$ be the weak poset Ramsey number. For every integer $n\geq n_0$ and $Q_n\in\mathcal Q$, the \emph{weak poset multiplicity} is defined as
\[
\operatorname M_{r,t}(Q_n\mid P_1,\dots,P_r)=\min_{\chi:\mathcal C_t(Q_n)\to[r]}\sum_{i=1}^r N_i(\chi),
\]
where $N_i(\chi)$ denotes the number of weak embeddings $f:P_i\hookrightarrow Q_n$ such that $f(P_i)$ is monochromatic of color $i$ under $\chi$.
\end{definition}

\begin{definition}
Let $\mathcal Q=\{Q_a:a\geq1\}$ be a family of finite posets with $Q_a\subseteq Q_{a+1}$ and $|Q_a|<|Q_{a+1}|$ for all $a\geq1$. Let $P_1,\dots,P_r$ be finite posets, and let $n_0=\operatorname R_{r,t}(\mathcal Q\mid P_1,\dots,P_r)$ be the weak poset Ramsey number. The \emph{weak poset Ramsey multiplicity} is defined as
\[
\operatorname{RM}_{r,t}(Q_{n_0}\mid P_1,\dots,P_r)=\min_{\substack{\chi:\mathcal C_t(Q_{n_0})\to[r]\\ \chi\text{ surjective}}}\sum_{i=1}^r N_i(\chi),
\]
where $N_i(\chi)$ is defined as in the weak poset multiplicity.
\end{definition}

If $P_1=\cdots=P_r=P$, we write $\operatorname M_{r,t}(Q_n\mid P)=\operatorname M_{r,t}(Q_n\mid P,\dots,P)$ and $\operatorname{RM}_{r,t}(Q_{n_0}\mid P)=\operatorname{RM}_{r,t}(Q_{n_0}\mid P,\dots,P)$. For strong poset Ramsey numbers $\operatorname R_{r,t}^{\sharp}(\mathcal Q\mid P_1,\dots,P_r)$, the corresponding strong poset multiplicity $\operatorname M_{r,t}^{\sharp}(Q_n\mid P_1,\dots,P_r)$ and strong poset Ramsey multiplicity $\operatorname{RM}_{r,t}^{\sharp}(Q_{n_0}\mid P_1,\dots,P_r)$ are defined analogously by replacing weak embeddings with induced subposets.

For positive integers $p$ and $t$, let $h_p(t)$ denote the number of distinct strict $t$-chains in $B_p$.

\begin{proposition}{\upshape \cite{KMOW25}}\label{th-La2-corrected}
For positive integers $p\geq1$ and $t\geq1$, the number of strict $t$-chains in $B_p$ is
\begin{equation}\label{eq-bl-hp-def-corrected}
h_p(t)=\sum_{j=0}^{t-1}(-1)^j\binom{t-1}{j}(t+1-j)^p.
\end{equation}
For fixed $t$, one has $h_p(t)=\Theta((t+1)^p)$.
\end{proposition}

For two integers $m$ and $N$ with $m\leq N$, let $a(m)$ be the number of distinct antichains in $B_m$, and let $g(m,N)$ be the number of strong embeddings of $B_m$ into $B_N$.

\begin{theorem}{\upshape\cite{AW17}}\label{number_b_n}
Let $m,N$ be integers with $m\leq N$, and let $g(m,N)$ be the number of strong embeddings of $B_m$ into $B_N$. Then
\[
\frac{N!}{(N-m)!}(a(m)-m)^{N-m}
\leq g(m,N)\leq
\frac{N!}{(N-m)!}a(m)^{N-m},
\]
where $a(m)$ is the number of antichains in $B_m$. It is known \cite{KM} that
\[
a(m)=2^{\binom{m}{\lfloor m/2\rfloor}(1+O(\log m/m))},
\]
and hence
\[
g(m,N)\leq 2^{m\log_2 N+(N-m)\binom{m}{\lfloor m/2\rfloor}(1+O(\log m/m))}.
\]
\end{theorem}

We now specialize the above multiplicity notions to the Boolean lattice family $\mathcal B=\{B_n:n\geq1\}$. The resulting weak Boolean multiplicity and weak Boolean Ramsey multiplicity are denoted by $\operatorname M_{r,t}(B_n\mid P_1,\dots,P_r)$ and $\operatorname{RM}_{r,t}(B_{n_0}\mid P_1,\dots,P_r)$. The corresponding strong Boolean multiplicity $\operatorname M_{r,t}^{\sharp}(B_n\mid P_1,\dots,P_r)$ and strong Boolean Ramsey multiplicity $\operatorname{RM}_{r,t}^{\sharp}(B_{n_0}\mid P_1,\dots,P_r)$ are defined analogously for induced copies.

\subsection{Our results}

This paper has two main parts. First, in the Boolean lattice, we study an arithmetic Schur-type subfamily of induced $3$-chains. Second, for general nested poset families, we prove lower and upper bounds for strong poset multiplicity, including a monotonicity result and a zero-multiplicity criterion.

\medskip\noindent\textbf{(1) Boolean lattice multiplicity for arithmetic Schur-type chains.}
In Section~\ref{sec_2}, we consider
\[
 E_n=\{(S,T,U)\in B_n^3:S\subsetneq T\subsetneq U,\ |S|+|T|=|U|\}.
\]
The rank equation $|S|+|T|=|U|$ makes this a Boolean-lattice analogue of Schur-type arithmetic configurations, rather than the full induced-copy problem for $C_3$. We obtain an exact enumeration of $E_n$, determine the two-color threshold for forcing a monochromatic member, and prove exponential lower and upper bounds for the minimum number of monochromatic arithmetic chains.

\begin{theorem}\label{thm:main}
Let
\[
 E_n=\{(S,T,U)\in B_n^3:S\subsetneq T\subsetneq U,\ |S|+|T|=|U|\}.
\]
Then the two-color threshold for forcing a monochromatic member of $E_n$ equals $9$. Moreover,
\[
|E_n|
=\binom{2n}{n}-[x^n](1+x+x^2)^n-2^n+1
=\frac{4^n}{\sqrt{\pi n}}\bigl(1+o(1)\bigr).
\]
If $M^{\mathrm{arith}}_2(B_n)$ denotes the minimum number of monochromatic members of $E_n$ over all two-colorings of $B_n$, then
\[
2^{\delta n+o(n)}
\leq M^{\mathrm{arith}}_2(B_n)
\leq 2^{\gamma n+o(n)},
\]
where
\[
\delta=-3\cdot\frac1{10}\log_2\frac1{10}-\frac7{10}\log_2\frac7{10}
      \approx 1.356779
\]
and
\[
\gamma=\min_{1/3\leq\alpha\leq1/2}
\max\{H(\alpha)+\alpha\log_2 3,\ 2H(1-2\alpha)\}
      \approx 1.567837.
\]
Here $H(x)=-x\log_2 x-(1-x)\log_2(1-x)$, with $H(0)=0$ by continuity, and the minimum is attained at the unique solution $\alpha_0\approx0.383292$ of
\[
H(\alpha)+\alpha\log_2 3=2H(1-2\alpha).
\]
\end{theorem}

\medskip\noindent\textbf{(2) Bounds for strong poset multiplicity.}
In Section~\ref{sec_3}, we first prove a double-counting lower bound for strong poset multiplicity.

\begin{theorem}\label{th-double-counting}
Let $\mathcal Q=\{Q_a:a\geq1\}$ be a family of finite posets with $Q_a\subseteq Q_{a+1}$ and $|Q_a|<|Q_{a+1}|$ for all $a\geq1$. Let $r,t\geq1$ be integers, let $P_1,\dots,P_r$ be finite posets, and let $n_0=\operatorname R^{\sharp}_{r,t}(\mathcal Q\mid P_1,\dots,P_r)$. For all integers $n\geq n_0$ and $Q_n\in\mathcal Q$,
\begin{equation}\label{eq:main_strong_bound}
\operatorname M^{\sharp}_{r,t}(Q_n\mid P_1,\dots,P_r)
\geq
\operatorname M^{\sharp}_{r,t}(Q_{n_0}\mid P_1,\dots,P_r)
\cdot \frac{\mathcal N^{\sharp}(Q_n,Q_{n_0})}{D^{\sharp}}.
\end{equation}
Here $D^{\sharp}$ denotes the maximum of $1$ and the numbers of strong copies $S\subseteq Q_n$ such that $S\cong Q_{n_0}$ and $P_i'\subseteq S$, as $i$ ranges over $[r]$ and $P_i'\subseteq Q_n$ ranges over strong copies of $P_i$.
\end{theorem}
We also give probabilistic upper bounds for strong poset multiplicity; see Proposition~\ref{pro-UUB}.

Let $\mathcal N^{\sharp}(Q_n,P)$ denote the total number of strong copies of $P$ in $Q_n$. 
We develop a Fourier--M\"obius method for
estimating strong poset multiplicity in an arbitrary finite host poset. Let
$P$ and $Q$ be finite posets, let
$h=\tau_t(P)=|\mathcal C_t(P)|\geq1$,
and assume that $Q$ contains a strong copy of $P$. For
$1\leq k\leq h$, let $\kappa_k(P,Q,t)$ be the Fourier overlap parameter
defined in Section~\ref{sec:fourier-posets}. Then
\begin{equation*}\label{eq:intro-fourier-multiplicity-bound}
\begin{aligned}
\operatorname M_{r,t}^{\sharp}(Q\mid P)
\geq\;&
\mathcal N^{\sharp}(Q,P)
\min_{\boldsymbol\alpha\in\Delta_{r-1}}
\Bigg[
\sum_{i=1}^r\alpha_i^h-
\sum_{i=1}^r\sum_{k=1}^h
\binom hk\kappa_k(P,Q,t)
\alpha_i^{h-k}
\bigl(\alpha_i(1-\alpha_i)\bigr)^{k/2}
\Bigg]_+,
\end{aligned}
\end{equation*}
where
\[
\Delta_{r-1}
=
\left\{
(\alpha_1,\ldots,\alpha_r)\in[0,1]^r:
\sum_{i=1}^r\alpha_i=1
\right\}
\]
and $[x]_+=\max\{x,0\}$.

\section{Boolean lattice multiplicity for arithmetic Schur-type chains}\label{sec_2}

In this section we do not study the general strong Boolean lattice multiplicity of all induced copies of $C_3$. Instead, we study the arithmetic subfamily
 $E_n=\{(S,T,U)\in B_n^3:\ S\subsetneq T\subsetneq U,\ |S|+|T|=|U|\}$.
The condition $|S|+|T|=|U|$ is an additional arithmetic constraint on the ranks, so this is a separate multiplicity problem and requires separate notation. Throughout this section we identify $B_n$ with $2^{[n]}$.

For a $2$-coloring $\chi:B_n\to\{0,1\}$, define
\[
 M_{\mathrm{arith}}(\chi)=|\{(S,T,U)\in E_n:\ \chi(S)=\chi(T)=\chi(U)\}|,
\]
and set
\[
 M^{\mathrm{arith}}_2(B_n)=\min_{\chi:B_n\to\{0,1\}}M_{\mathrm{arith}}(\chi).
\]

\subsection{Enumeration of arithmetic Schur-type chains}

Recall that the number of all induced copies of $C_3$ in $B_n$ is $h_n(3)=4^n-2\cdot3^n+2^n$, by Proposition~\ref{th-La2-corrected}. We now count the arithmetic subfamily $E_n$.

\begin{proof}[Proof of Theorem \ref{thm:main}]
Fix ranks $|S|=k$, $|T|=m$, and $|U|=k+m$. Since
$S\subsetneq T\subsetneq U$,
we must have $1\leq k<m$, and the condition $|U|=k+m$ gives $k+m\leq n$.

For fixed $k$ and $m$, we count the triples $(S,T,U)$ as follows. First choose the middle set $T$ in $\binom{n}{m}$ ways. Then choose $S\subset T$ with $|S|=k$ in $\binom{m}{k}$ ways. Finally, since $U$ must contain $T$ and have size $k+m$, the set $U\setminus T$ must be a $k$-subset of $[n]\setminus T$, and this can be chosen in $\binom{n-m}{k}$ ways. Hence
\[
|E_n|
=
\sum_{\substack{1\leq k<m\\ k+m\leq n}}
\binom{n}{m}\binom{m}{k}\binom{n-m}{k}.
\]
We now transform this sum. Using
\[
\binom{n}{m}\binom{m}{k}
=
\binom{n}{k}\binom{n-k}{m-k},
\]
we obtain
\begin{align*}
\sum_{m=k+1}^{n-k}\binom{n}{m}\binom{m}{k}\binom{n-m}{k}
&=
\binom{n}{k}
\sum_{m=k+1}^{n-k}
\binom{n-k}{m-k}\binom{n-m}{k}.
\end{align*}
Let $\ell=m-k$. Then $1\leq \ell\leq n-2k$, and therefore
\begin{align*}
\sum_{m=k+1}^{n-k}\binom{n}{m}\binom{m}{k}\binom{n-m}{k}
&=
\binom{n}{k}
\sum_{\ell=1}^{n-2k}
\binom{n-k}{\ell}\binom{n-k-\ell}{k}  \\
&=
\binom{n}{k}\binom{n-k}{k}
\sum_{\ell=1}^{n-2k}
\binom{n-2k}{\ell}  \\
&=
\binom{n}{k}\binom{n-k}{k}
\bigl(2^{n-2k}-1\bigr).
\end{align*}
Summing over $1\leq k\leq \lfloor n/2\rfloor$ gives
\[
|E_n|
=
\sum_{k=1}^{\lfloor n/2\rfloor}
\binom{n}{k}\binom{n-k}{k}
\bigl(2^{n-2k}-1\bigr).
\]
It remains to derive the equivalent form. Let
\[
A_n=
\sum_{k=0}^{\lfloor n/2\rfloor}
\binom{n}{k}\binom{n-k}{k}2^{n-2k}
\]
and
\[
T_n=
\sum_{k=0}^{\lfloor n/2\rfloor}
\binom{n}{k}\binom{n-k}{k}.
\]
Then the preceding formula gives
$|E_n|=A_n-T_n-2^n+1$,
because the $k=0$ term of $A_n$ is $2^n$, while the $k=0$ term of $T_n$ is $1$.

We claim that
$A_n=\binom{2n}{n}$.
Indeed, consider a set of \(2n\) distinct elements divided into \(n\) fixed
pairs. We count the number of ways to choose \(n\) elements from this set.
Clearly, the answer is \(\binom{2n}{n}\).
We count the same quantity according to the number of pairs from which both
elements are chosen. Suppose exactly \(k\) pairs contribute both of their
elements. Since the chosen set has size \(n\), exactly \(k\) pairs must
contribute no element, and the remaining \(n-2k\) pairs contribute exactly
one element each. The \(k\) pairs contributing two elements can be chosen in
\(\binom{n}{k}\) ways. Then the \(k\) pairs contributing no element can be
chosen in \(\binom{n-k}{k}\) ways. Finally, from each of the remaining
\(n-2k\) pairs, one of the two elements is chosen, giving \(2^{n-2k}\)
choices. Hence, for fixed \(k\), the number of choices is
\[
\binom{n}{k}\binom{n-k}{k}2^{n-2k}.
\]
Summing over all possible \(k\), we obtain
\[
A_n
=
\sum_{k=0}^{\lfloor n/2\rfloor}
\binom{n}{k}\binom{n-k}{k}2^{n-2k}
=
\binom{2n}{n}.
\]

Therefore
\[
|E_n|
=
\binom{2n}{n}
-
T_n
-
2^n+1
=
\binom{2n}{n}
-
\sum_{k=0}^{\lfloor n/2\rfloor}
\binom{n}{k}\binom{n-k}{k}
-
2^n+1.
\]

Finally, we estimate the error term. The quantity $T_n$ counts ordered pairs $(A,B)$ of disjoint subsets of $[n]$ with $|A|=|B|$. Hence $T_n$ is at most the number of all ordered pairs of disjoint subsets of $[n]$, which is $3^n$, since each element of $[n]$ may lie in $A$, lie in $B$, or lie in neither. Thus
$T_n\leq 3^n$.
Consequently,
\[
|E_n|
=
\binom{2n}{n}+O(3^n).
\]
Using the standard asymptotic formula
\[
\binom{2n}{n}
=
\frac{4^n}{\sqrt{\pi n}}\bigl(1+o(1)\bigr),
\]
and noting that $3^n=o(4^n/\sqrt n)$, we obtain that 
\[
|E_n|
=
\frac{4^n}{\sqrt{\pi n}}\bigl(1+o(1)\bigr).
\]
\end{proof}

\subsection{Threshold and exponential multiplicity bounds}

\begin{proposition}\label{prop:global-bounds}
As $n\to\infty$,
\[
2^{\delta n+o(n)}
\leq M^{\mathrm{arith}}_2(B_n)
\leq 2^{\gamma n+o(n)},
\]
where
\[
\delta=-3\cdot\frac1{10}\log_2\frac1{10}-\frac7{10}\log_2\frac7{10}
      \approx 1.356779
\]
and
\[
\gamma=\min_{1/3\leq\alpha\leq1/2}
\max\{H(\alpha)+\alpha\log_2 3,\ 2H(1-2\alpha)\}
      \approx 1.567837.
\]
Here \(H(x)=-x\log_2x-(1-x)\log_2(1-x)\), with \(H(0)=0\) by continuity, is the binary entropy function.
\end{proposition}

\begin{proof}
We first prove the upper bound. Fix a real number
\(\alpha\in[1/3,1/2]\). Color a set \(A\in B_n\) blue if
$\alpha n<|A|\leq 2\alpha n$,
and color it red otherwise. The boundary inequalities are interpreted with the obvious integer rounding; these roundings affect only the \(2^{o(n)}\) factors below. There is no blue member of \(E_n\). Indeed, if
\(|S|=k\), \(|T|=m\), \(|U|=k+m\), and \(S,T,U\) are all blue, then
\(k>\alpha n\) and \(m>\alpha n\), so \(k+m>2\alpha n\), contradicting the
blue condition for \(U\).

It remains to count possible red triples. For a red triple the middle rank \(m=|T|\) cannot satisfy \(\alpha n<m\leq2\alpha n\), since then \(T\) would be blue. Thus either \(m\leq\alpha n\) or \(m>2\alpha n\). If \(m\leq\alpha n\), then red
monochromaticity of \(U\) forces \(k+m\leq\alpha n\). For each
\(u\leq\alpha n\), the number of chains \(S\subsetneq T\subsetneq U\) with
\(|U|=u\) is at most \(3^u\binom nu\). Hence this part contributes at most
\[
\sum_{u\leq \alpha n}\binom nu3^u
\leq
2^{(H(\alpha)+\alpha\log_2 3)n+o(n)}.
\]
If \(m>2\alpha n\), put \(q=n-m\). Then \(q<(1-2\alpha)n\), and
\(k\leq q\). For fixed \(q\), the number of choices is at most
\[
\binom nq\sum_{k=1}^{q}\binom{n-q}{k}\binom qk
\leq \binom nq^2,
\]
where the last inequality follows from Vandermonde's identity.  Summing over
\(q\leq(1-2\alpha)n\) gives at most
\[
\sum_{q\leq(1-2\alpha)n}\binom nq^2
\leq
2^{2H(1-2\alpha)n+o(n)}.
\]
Thus, for every \(\alpha\in[1/3,1/2]\), we have 
\[
M^{\mathrm{arith}}_2(B_n)
\leq
2^{\max\{H(\alpha)+\alpha\log_2 3,\ 2H(1-2\alpha)\}n+o(n)}.
\]
The first function in the maximum is increasing on \([1/3,1/2]\), while the
second is decreasing. At \(\alpha=1/3\) the first is smaller than the second,
and at \(\alpha=1/2\) the first is larger than the second. Therefore the
optimum is attained at their unique intersection. This gives the exponent
\[
\gamma
=
\min_{1/3\leq\alpha\leq1/2}
\max\{H(\alpha)+\alpha\log_2 3,\ 2H(1-2\alpha)\}
\approx
1.567837.
\]

We now prove the lower bound. For all sufficiently large \(n\), put \(q=\lfloor n/10\rfloor\), and let
\(\mathfrak C_n\) be the set of maximal chains of \(B_n\). The choice of the factor \(1/10\) is used only in the final entropy comparison below. For a maximal chain
$C:\ \emptyset=A_0\subsetneq A_1\subsetneq\cdots\subsetneq A_n=[n]$,
look at the nine elements \(A_q,A_{2q},\dots,A_{9q}\). The induced coloring of
these nine elements gives a two-coloring of \(\{1,\dots,9\}\).  Since the weak
Schur number satisfies \(WS(2)=8\), there are integers \(1\leq i<j\) with
\(i+j\leq9\) such that
$\chi(A_{iq})=\chi(A_{jq})=\chi(A_{(i+j)q})$.
Thus every maximal chain contains at least one monochromatic arithmetic triple
whose ranks are of the form \((iq,jq,(i+j)q)\).

Fix ranks \(a<b\) with \(a+b\leq n\). Let
\[
N_{a,b}=\binom{n}{b}\binom ba\binom{n-b}{a}
\]
be the number of arithmetic triples \((S,T,U)\in E_n\) with
\(|S|=a\), \(|T|=b\), and \(|U|=a+b\). A fixed such triple is contained in
exactly
$c_{a,b}=a!(b-a)!a!(n-a-b)!$
maximal chains, and hence
$N_{a,b}c_{a,b}=n!$.

Now sum, over all maximal chains, the weight \(1/c_{a,b}\) for each
monochromatic arithmetic triple of ranks \((a,b,a+b)\) lying on the chain.
Every concrete monochromatic triple contributes total weight \(1\), so this
weighted sum is exactly \(M_{\mathrm{arith}}(\chi)\). Since each maximal chain
contains one monochromatic triple of ranks \((iq,jq,(i+j)q)\), we get
\[
M_{\mathrm{arith}}(\chi)
\geq
n!\min_{\substack{1\leq i<j\\ i+j\leq9}}\frac1{c_{iq,jq}}
=
\min_{\substack{1\leq i<j\\ i+j\leq9}}N_{iq,jq}.
\]
By Stirling's formula,
\[
N_{iq,jq}
=
2^{\left[
-\frac{i}{10}\log_2\frac{i}{10}
-\frac{j-i}{10}\log_2\frac{j-i}{10}
-\frac{i}{10}\log_2\frac{i}{10}
-\left(1-\frac{i+j}{10}\right)\log_2\left(1-\frac{i+j}{10}\right)
+o(1)\right]n}.
\]
A direct comparison of the finitely many admissible pairs \(1\leq i<j\) with
\(i+j\leq9\) shows that the minimum is attained at
\((i,j)=(1,2)\) and \((i,j)=(1,8)\). Equivalently, the four parts in the corresponding multinomial coefficient have normalized sizes \(1/10,1/10,1/10,7/10\). This gives
\[
\delta=
-3\cdot\frac1{10}\log_2\frac1{10}
-\frac7{10}\log_2\frac7{10}
\approx 1.356779.
\]
Therefore, \(M_{\mathrm{arith}}(\chi)\geq2^{\delta n+o(n)}\) for every
two-coloring \(\chi\), which proves the lower bound.
\end{proof}
Combining Theorem~\ref{thm:main} and Proposition~\ref{prop:global-bounds}, we obtain the following theorem.
\begin{theorem}\label{thm:arith-main-section}
The arithmetic threshold equals $9$, the number of arithmetic induced $3$-chains satisfies
\[
|E_n|=\frac{4^n}{\sqrt{\pi n}}\bigl(1+o(1)\bigr),
\]
and
\[
2^{\delta n+o(n)}
\leq M^{\mathrm{arith}}_2(B_n)
\leq 2^{\gamma n+o(n)}.
\]
\end{theorem}

\section{Bounds for the poset multiplicity}\label{sec_3}

Let $r,t\geq1$ be integers, and let $\mathcal Q=\{Q_a:a\geq1\}$ be a family of finite posets such that $Q_a\subseteq Q_{a+1}$ and $|Q_a|<|Q_{a+1}|$ for all $a\geq1$. Recall that $\mathcal C_t(Q)$ is the set of all $t$-chains in a poset $Q$. For a finite poset $P$, let $\mathcal N(Q_n,P)$ be the number of weak copies of $P$ in $Q_n$.

\subsection{Bounds for strong poset multiplicity}

For posets $Q,R$ and a strong subposet $P'\subseteq Q$, define:
\begin{itemize}
\item $\mathcal N^\sharp(Q,R)$: the number of strong subposets of $Q$ isomorphic to $R$;
\item $\mathcal N^\sharp(Q,P';R)$: the number of strong subposets of $Q$ isomorphic to $R$ that contain $P'$.
\end{itemize}
Let $P_1,\dots,P_r$ be finite posets, let $n_0=\operatorname R^\sharp_{r,t}(\mathcal Q\mid P_1,\dots,P_r)$ be the strong poset Ramsey number, and set
\[
s_0^\sharp=\operatorname M^\sharp_{r,t}(Q_{n_0}\mid P_1,\dots,P_r).
\]
For a fixed $n\geq n_0$, define
\[
D^\sharp=\max\left\{1,\mathcal N^\sharp(Q_n,P_i';Q_{n_0})\mid 1\leq i\leq r,\ P_i'\subseteq Q_n,\ P_i'\cong P_i,\ P_i'\text{ strong}\right\}.
\]

\begin{proof}[Proof of Theorem~\ref{th-double-counting}]
Let $n\geq n_0$, and let $\chi:\mathcal C_t(Q_n)\to[r]$ be an arbitrary $r$-coloring. Let $\mathcal S^\sharp$ be the set of strong copies $S\subseteq Q_n$ with $S\cong Q_{n_0}$, so $|\mathcal S^\sharp|=\mathcal N^\sharp(Q_n,Q_{n_0})$.

For each $S\in\mathcal S^\sharp$, let $\chi_S$ be the restriction of $\chi$ to $\mathcal C_t(S)$. By the definition of $s_0^\sharp$ as a minimum over all $r$-colorings, not necessarily exact ones,
\[
\sum_{i=1}^r N_i^\sharp(\chi_S)\geq s_0^\sharp
\]
for all $S\in\mathcal S^\sharp$. Summing over $S$ yields
\[
\sum_{S\in\mathcal S^\sharp}\sum_{i=1}^rN_i^\sharp(\chi_S)
\geq s_0^\sharp\mathcal N^\sharp(Q_n,Q_{n_0}).
\]
The left-hand side counts pairs $(S,P_i')$, where $S\in\mathcal S^\sharp$, $P_i'\subseteq S$, $P_i'\cong P_i$, and $P_i'$ is monochromatic of color $i$. Hence it is at most
\[
D^\sharp\sum_{i=1}^rN_i^\sharp(\chi).
\]
Therefore
\[
\sum_{i=1}^rN_i^\sharp(\chi)
\geq s_0^\sharp\frac{\mathcal N^\sharp(Q_n,Q_{n_0})}{D^\sharp}.
\]
Taking the minimum over all $\chi$ proves the theorem.
\end{proof}

For a finite poset $P$, write
$\tau_t(P)=|\mathcal C_t(P)|$
for the number of $t$-chains in $P$. In the following results we assume $\tau_t(P)\geq1$.

\begin{proposition}\label{pro-UUB}
For any integer $n\geq1$ and $Q_n\in\mathcal Q$,
\[
\operatorname M_{r,t}^\sharp(Q_n\mid P)
\leq\frac{\mathcal N^\sharp(Q_n,P)}{r^{\tau_t(P)-1}}.
\]
\end{proposition}

\begin{proof}
Let $\chi:\mathcal C_t(Q_n)\to[r]$ be a uniform independent random $r$-coloring of the $t$-chains of $Q_n$. For each induced copy $F$ of $P$ in $Q_n$, let $A_F$ be the event that all $t$-chains of $F$ have one common color. Since $F\cong P$, it contains exactly $\tau_t(P)$ $t$-chains. Hence
\[
\mathbb P(A_F)=r\left(\frac1r\right)^{\tau_t(P)}=\frac1{r^{\tau_t(P)-1}}.
\]
Let $X$ be the number of monochromatic induced copies of $P$. By linearity of expectation,
\[
\mathbb E (X)=\mathcal N^\sharp(Q_n,P)\cdot\frac1{r^{\tau_t(P)-1}}.
\]
Therefore some coloring has at most this many monochromatic induced copies, and the desired upper bound follows.
\end{proof}

\begin{corollary}
For the Boolean lattice family $\mathcal Q=\mathcal B=\{B_n:n\geq1\}$,
\[
\operatorname M_{r,t}^\sharp(B_n\mid P)
\leq\frac{\mathcal N^\sharp(B_n,P)}{r^{\tau_t(P)-1}}.
\]
\end{corollary}

\begin{corollary}
If $t=1$ and $s=|P|$, then $\tau_1(P)=s$, and
\[
\operatorname M_r^\sharp(Q_n\mid P)
\leq\frac{\mathcal N^\sharp(Q_n,P)}{r^{s-1}}.
\]
\end{corollary}

\subsection{A Fourier--M\"obius framework for finite posets}\label{sec:fourier-posets}
This subsection combines two coordinate systems on a finite poset: zeta--M\"obius
coordinates, which encode the order relation, and Laplacian Fourier
coordinates, which provide orthogonality and Parseval's identity. The counting identities require
only an orthonormal basis containing the constant function, while a Laplacian
eigenbasis gives the coefficients a spectral interpretation; see
\cite{Chung97,ODonnell14,StanleyEC1}.
We write each strict \(t\)-chain as an ordered tuple
\((x_1<\cdots<x_t)\).

\subsubsection{M\"obius coordinates and graph-Fourier coordinates}

Let \(R\) be a finite poset. We write \(\mathbb R^R\) for the real
vector space of all functions \(f:R\to\mathbb R\). The zeta operator
\(Z_R:\mathbb R^R\to\mathbb R^R\) is defined by
\((Z_Ru)(x)=\sum_{y\le_R x}u(y)\). Thus \(Z_Ru\) records the cumulative
sum of \(u\) over each principal order ideal \(\downarrow x=\{y\in R:y\le_R x\}\).

Choose a linear extension \(r_1,\ldots,r_N\) of \(R\), where \(N=|R|\).
With respect to this ordering, the matrix of \(Z_R\) is triangular with
all diagonal entries equal to \(1\). Hence \(Z_R\) is invertible. The
\textit{M\"obius function} \(\mu_R\) is the kernel of \(Z_R^{-1}\). Equivalently,
it is the unique function on comparable pairs satisfying
\(\mu_R(x,x)=1\) and, for \(x<_R y\),
\(\mu_R(x,y)=-\sum_{x\le_R z<_R y}\mu_R(x,z)\). This recurrence is
equivalent to
\begin{equation*}
 \sum_{x\le_R z\le_R y}\mu_R(x,z)
 =
 \begin{cases}
  1,& x=y,\\
  0,& x<_R y.
 \end{cases}
\end{equation*}

For \(f\in\mathbb R^R\), define its \textit{M\"obius increment} by
$$(\partial_R f)(x)=(Z_R^{-1}f)(x)=\sum_{y\le_R x}\mu_R(y,x)f(y).$$
If \(d=\partial_Rf\), then \(d=Z_R^{-1}f\), and hence \(f=Z_Rd\). Therefore,
for every \(x\in R\),
\begin{equation}\label{eq:mobius-reconstruction}
 f(x)=(Z_Rd)(x)=\sum_{y\le_R x}d(y)
      =\sum_{y\le_R x}(\partial_Rf)(y).
\end{equation}
This is the M\"obius inversion formula. Equivalently, \((\partial_Rf)(y)\)
is the order-localized contribution of \(y\), and \(f(x)\) is reconstructed
by summing all such contributions over the ideal \(\downarrow x\).

Let \(G_R=(R,E_R)\) be the undirected cover graph of \(R\). Thus
\(\{x,y\}\in E_R\) if and only if either \(x\lessdot_R y\) or
\(y\lessdot_R x\), where \(x\lessdot_R y\) means that \(x<_R y\) and no
\(z\in R\) satisfies \(x<_R z<_R y\). We equip \(\mathbb R^R\) with the
normalized inner product $$\langle f,g\rangle_R=|R|^{-1}\sum_{x\in R}f(x)g(x),$$
and we write \(\|f\|_{2,R}^2=\langle f,f\rangle_R\). The \emph{unnormalized
cover-graph Laplacian} is defined by
$$(\Delta_Rf)(x)=\sum_{y:\{x,y\}\in E_R}(f(x)-f(y)).$$

\begin{lemma}\label{lem:green-identity}
For all \(f,g\in\mathbb R^R\),
\begin{equation}\label{eq:green-identity}
 \langle f,\Delta_R g\rangle_R
 =
 \frac1{|R|}
 \sum_{\{x,y\}\in E_R}
 \bigl(f(x)-f(y)\bigr)\bigl(g(x)-g(y)\bigr).
\end{equation}
Consequently, \(\Delta_R\) is self-adjoint and positive semidefinite.
\end{lemma}

\begin{proof}
By the definitions of the inner product and the Laplacian,
\[
 \langle f,\Delta_Rg\rangle_R
 =
 \frac1{|R|}
 \sum_{x\in R}\sum_{y:\{x,y\}\in E_R}
 f(x)\bigl(g(x)-g(y)\bigr).
\]
In this double sum, each unordered edge \(\{x,y\}\in E_R\) appears twice.
The two terms contributed by this edge are
\(f(x)(g(x)-g(y))\) and \(f(y)(g(y)-g(x))\), whose sum is
\((f(x)-f(y))(g(x)-g(y))\). Regrouping the double sum by unordered edges
gives \eqref{eq:green-identity}. The right-hand side of
\eqref{eq:green-identity} is symmetric in \(f\) and \(g\), so
\(\langle f,\Delta_Rg\rangle_R=\langle \Delta_Rf,g\rangle_R\). Thus
\(\Delta_R\) is self-adjoint. Taking \(g=f\) gives
$$\langle f,\Delta_Rf\rangle_R=|R|^{-1}\sum_{\{x,y\}\in E_R}(f(x)-f(y))^2\ge0,
$$
so \(\Delta_R\) is positive semidefinite.
\end{proof}

\begin{lemma}\label{lem:zero-eigenspace}
Let \(R_1,\ldots,R_c\) be the connected components of the cover graph
\(G_R\). Then
$\ker\Delta_R
 =
 \operatorname{span}\{\mathbf1_{R_1},\ldots,\mathbf1_{R_c}\}$.
In particular, \(\dim\ker\Delta_R=c\). Hence \(G_R\) is connected if and
only if \(\ker\Delta_R=\operatorname{span}\{\mathbf1_R\}\).
\end{lemma}
\begin{proof}
Suppose first that \(f\in\ker\Delta_R\). Then \(\Delta_Rf=0\), so by
Lemma~\ref{lem:green-identity},
\[
 0=\langle f,\Delta_Rf\rangle_R
 =
 \frac1{|R|}\sum_{\{x,y\}\in E_R}(f(x)-f(y))^2.
\]
Every summand is nonnegative, hence every summand is zero. Thus
\(f(x)=f(y)\) for every edge \(\{x,y\}\in E_R\). If two vertices lie in
the same connected component, they are joined by a path, and the preceding
edge equality along the path implies that \(f\) has the same value at the
two vertices. Hence \(f\) is constant on each connected component.

Conversely, if \(f\) is constant on each connected component, then every
edge \(\{x,y\}\in E_R\) has both endpoints in the same component, so
\(f(x)=f(y)\). Therefore, each summand in $$(\Delta_Rf)(x)=
\sum_{y:\{x,y\}\in E_R}(f(x)-f(y))$$ is zero, and hence \(\Delta_Rf=0\).
Thus the kernel consists exactly of functions that are constant on each
connected component. Such a function is uniquely of the form
\(a_1\mathbf1_{R_1}+\cdots+a_c\mathbf1_{R_c}\), and the indicator functions
\(\mathbf1_{R_1},\ldots,\mathbf1_{R_c}\) are linearly independent.
\end{proof}

Since \(\Delta_R\) is self-adjoint on the finite-dimensional real inner
product space \(\mathbb R^R\), the spectral theorem gives an orthonormal
basis of real eigenvectors. Since \(\Delta_R\) is positive semidefinite,
all eigenvalues are nonnegative. We choose such a basis
\(\psi_0,\psi_1,\ldots,\psi_{|R|-1}\) with
\(\Delta_R\psi_j=\lambda_j\psi_j\), ordered so that
\(0=\lambda_0\le\lambda_1\le\cdots\le\lambda_{|R|-1}\), and with
\(\psi_0\equiv1\). This is possible because \(\Delta_R1=0\), and because
\(\|1\|_{2,R}=1\) under the normalized inner product.

If \(G_R\) is disconnected, then it follows from Lemma~\ref{lem:zero-eigenspace} that the
zero eigenspace has dimension greater than \(1\). In that case, after
choosing \(\psi_0=1\), we choose the remaining zero-eigenvalue vectors
orthogonal to \(\psi_0\). Completing with orthonormal eigenvectors from
the positive eigenspaces gives
\[
 \operatorname{span}\{\psi_j:1\le j\le |R|-1\}
 =
 1^\perp
 =
 \{f\in\mathbb R^R:\langle f,1\rangle_R=0\}.
\]
Thus the condition \(j\ge1\) always means orthogonality to the global
constant function, not necessarily positivity of the eigenvalue. In a
disconnected cover graph, some vectors with \(j\ge1\) may still satisfy
\(\lambda_j=0\).

For \(f\in\mathbb R^R\), define its \textit{Fourier coefficient} by
\(\widehat f(j)=\langle f,\psi_j\rangle_R\). Since
\(\psi_0,\ldots,\psi_{|R|-1}\) is an orthonormal basis, we have
\begin{equation}\label{eq:fourier-parseval}
 f=\sum_{j=0}^{|R|-1}\widehat f(j)\psi_j,
 \qquad
 \|f\|_{2,R}^2=\sum_{j=0}^{|R|-1}|\widehat f(j)|^2.
\end{equation}
Moreover, \(\widehat f(0)=\langle f,1\rangle_R=|R|^{-1}\sum_{x\in R}f(x)\),
and hence
 $$\bigl\|f-\widehat f(0)1\bigr\|_{2,R}^2
 =
\sum_{j=1}^{|R|-1}|\widehat f(j)|^2.$$
The Dirichlet energy has the spectral expansion
$$
 \langle f,\Delta_Rf\rangle_R
 =
\sum_{j=0}^{|R|-1}\lambda_j|\widehat f(j)|^2.
$$
Indeed, this follows by applying \(\Delta_R\psi_j=\lambda_j\psi_j\) to
the Fourier expansion and then using orthonormality.

The M\"obius and Fourier coordinates are related explicitly. For
\(0\le j\le |R|-1\) and \(y\in R\), set
\(\beta_j(y)=|R|^{-1}\sum_{x\ge_R y}\psi_j(x)\). Using
\eqref{eq:mobius-reconstruction}, we get
\[
 \widehat f(j)
 =
 \frac1{|R|}\sum_{x\in R}f(x)\psi_j(x)
 =
 \sum_{y\in R}(\partial_Rf)(y)
 \left(\frac1{|R|}\sum_{x\ge_R y}\psi_j(x)\right)
 =
 \sum_{y\in R}(\partial_Rf)(y)\beta_j(y).
\]
Thus graph-Fourier coefficients are fixed linear combinations of the
order-localized M\"obius increments. In the estimates below we use the
orthogonal coefficients \(\widehat f(j)\), because they satisfy Parseval's
identity.

For \(R=B_n\), identified with \(\{0,1\}^n\), one may choose the Walsh
characters \(\psi_A(S)=(-1)^{|A\cap S|}\), where \(A,S\subseteq[n]\).
They form an orthonormal eigenbasis and satisfy
\(\Delta_{B_n}\psi_A=2|A|\psi_A\). For a coloring of \(t\)-chains in a
finite poset \(Q\), however, the relevant underlying poset is
\(X=\mathcal C_t(Q)\), ordered coordinatewise: if
\(C=(x_1<\cdots<x_t)\) and \(C'=(y_1<\cdots<y_t)\), then
\(C\preceq C'\) if and only if \(x_s\le_Q y_s\) for every \(s\in[t]\).
All Fourier coefficients below are taken on this chain poset \(X\) with
the normalized uniform inner product. When \(t=1\) and \(Q=B_n\), this is
the usual Walsh-Fourier expansion on \(B_n\). When \(t\ge2\), the poset
\(\mathcal C_t(B_n)\) is a different poset, and the resulting eigenbasis
should not be identified with the classical Walsh basis on \(B_n\).

\subsubsection{A lower bound for multiplicity}

Let \(P\) and \(Q\) be finite posets, and let \(t\ge1\). Put
\(Y=\mathcal C_t(P)\), \(h=|Y|=\tau_t(P)\ge1\), \(X=\mathcal C_t(Q)\), and
\(m=|X|\). Assume that \(P\) has at least one strong embedding into \(Q\).
Let \(\operatorname{Emb}^{\sharp}(P,Q)\) be the set of strong embeddings
\(\phi:P\to Q\), and write \(e^\sharp(P,Q)=|\operatorname{Emb}^{\sharp}(P,Q)|\).
Let \(a(P)=|\operatorname{Aut}(P)|\). Since strong copies are counted by
their images, and since every strong image copy of \(P\) has exactly
\(a(P)\) parametrizations by strong embeddings, we have
\(e^\sharp(P,Q)=a(P)\mathcal N^\sharp(Q,P)\).

For \(D=(x_1<\cdots<x_t)\in Y\) and
\(\phi\in\operatorname{Emb}^{\sharp}(P,Q)\), define
\(\phi_*D=(\phi(x_1)<\cdots<\phi(x_t))\in X\). Since a strong embedding
preserves and reflects order,it follows that \(\mathcal C_t(\phi(P))=\{\phi_*D:D\in
\mathcal C_t(P)\}\). Indeed, every chain in \(P\) maps to a chain in
\(\phi(P)\), and conversely any chain in \(\phi(P)\) pulls back to a chain
in \(P\) because order relations in the image are reflected by \(\phi\).

Choose a real orthonormal eigenbasis \(\psi_0,\ldots,\psi_{m-1}\) on
\(\mathbb R^X\) with \(\psi_0\equiv1\) and
\(\operatorname{span}\{\psi_j:j\ge1\}=1^\perp\). Let
\(\chi:X\to[r]\) be an arbitrary coloring. For \(i\in[r]\), define the
color indicator \(f_i(C)=\mathbf1_{\{\chi(C)=i\}}\) and its density
$$\alpha_i=\langle f_i,1\rangle_X=m^{-1}\sum_{C\in X}f_i(C).$$ Then
\(\widehat f_i(0)=\alpha_i\), \(\sum_{i=1}^r f_i=1\), and
\(\sum_{i=1}^r\alpha_i=1\). Let \(N_i^\sharp(\chi)\) be the number of
strong image copies \(F\subseteq Q\) of \(P\) for which every chain in
\(\mathcal C_t(F)\) has color \(i\).

Let \(I_m=\{0,1,\ldots,m-1\}\). For a multi-index
\(\mathbf j=(j_D)_{D\in Y}\in I_m^Y\), set
\(\operatorname{supp}(\mathbf j)=\{D\in Y:j_D\ne0\}\). Define the \textit{normalized pattern kernel}
$$ \mathcal K_{P,Q,t}(\mathbf j)
 =
 \frac1{e^\sharp(P,Q)}
\sum_{\phi\in\operatorname{Emb}^{\sharp}(P,Q)}
 \prod_{D\in Y}\psi_{j_D}(\phi_*D).$$

\begin{theorem}
\label{thm:poset-fourier-expansion}
For every coloring \(\chi:\mathcal C_t(Q)\to[r]\),
\begin{equation}\label{eq:poset-fourier-copy-count}
 \frac{\sum_{i=1}^rN_i^\sharp(\chi)}{\mathcal N^\sharp(Q,P)}
 =
 \sum_{i=1}^r\sum_{\mathbf j\in I_m^Y}
 \mathcal K_{P,Q,t}(\mathbf j)
 \prod_{D\in Y}\widehat f_i(j_D).
\end{equation}
If \(\mathbf0\) denotes the all-zero multi-index, then
\(\mathcal K_{P,Q,t}(\mathbf0)=1\), and therefore
\[
 \frac{\sum_{i=1}^rN_i^\sharp(\chi)}{\mathcal N^\sharp(Q,P)}
 =
 \sum_{i=1}^r\alpha_i^h
 +
 \operatorname{Err}_{P,Q,t}(\chi),
\]
where
\[
 \operatorname{Err}_{P,Q,t}(\chi)
 =
 \sum_{i=1}^r
 \sum_{\substack{\mathbf j\in I_m^Y\\ \mathbf j\ne\mathbf0}}
 \mathcal K_{P,Q,t}(\mathbf j)
 \prod_{D\in Y}\widehat f_i(j_D).
\]
\end{theorem}

\begin{proof}
Fix \(i\in[r]\). For a strong embedding \(\phi\), the product
\(\prod_{D\in Y}f_i(\phi_*D)\) is equal to \(1\) if all members of
\(\mathcal C_t(\phi(P))\) have color \(i\), and is equal to \(0\) otherwise.
Since each strong image copy has exactly \(a(P)\) parametrizations by
strong embeddings, we have
\[
 \frac1{e^\sharp(P,Q)}
 \sum_{\phi\in\operatorname{Emb}^{\sharp}(P,Q)}
 \prod_{D\in Y}f_i(\phi_*D)
 =
 \frac{N_i^\sharp(\chi)}{\mathcal N^\sharp(Q,P)}.
\]
For every \(C\in X\), the Fourier expansion gives
\(f_i(C)=\sum_{j=0}^{m-1}\widehat f_i(j)\psi_j(C)\). Substituting this
identity into the preceding average, expanding the finite product over
\(D\in Y\), and interchanging finite sums gives
\[
 \frac{N_i^\sharp(\chi)}{\mathcal N^\sharp(Q,P)}
 =
 \sum_{\mathbf j\in I_m^Y}
 \left(
 \frac1{e^\sharp(P,Q)}
 \sum_{\phi\in\operatorname{Emb}^{\sharp}(P,Q)}
 \prod_{D\in Y}\psi_{j_D}(\phi_*D)
 \right)
 \prod_{D\in Y}\widehat f_i(j_D).
\]
The expression in parentheses is \(\mathcal K_{P,Q,t}(\mathbf j)\). Summing
over \(i\) proves \eqref{eq:poset-fourier-copy-count}. If
\(\mathbf j=\mathbf0\), then every factor \(\psi_{j_D}\) is \(\psi_0=1\),
so \(\mathcal K_{P,Q,t}(\mathbf0)=1\). The corresponding coefficient
product is \(\prod_{D\in Y}\widehat f_i(0)=\alpha_i^h\). Separating this
all-zero term from all remaining multi-indices gives the stated error
decomposition.
\end{proof}

For every nonempty \(A\subseteq Y\), define
\begin{equation}\label{eq:kappa-A-def}
 \kappa_A(P,Q,t)^2
 =
 \sum_{\substack{\mathbf j\in I_m^Y\\ \operatorname{supp}(\mathbf j)=A}}
 |\mathcal K_{P,Q,t}(\mathbf j)|^2.
\end{equation}
Thus \(\kappa_A\) measures the \(L^2\)-mass of the pattern kernel on those
Fourier modes in which precisely the coordinates indexed by \(A\) lie in
\(1^\perp\), while the coordinates outside \(A\) are fixed at the global
constant mode.

\begin{proposition}
\label{prop:poset-fourier-parseval-bound}
For every coloring \(\chi:\mathcal C_t(Q)\to[r]\),
\begin{equation}\label{eq:fourier-parseval-error-bound}
 \left|
 \frac{\sum_{i=1}^rN_i^\sharp(\chi)}{\mathcal N^\sharp(Q,P)}
 -
 \sum_{i=1}^r\alpha_i^h
 \right|
 \le
 \sum_{i=1}^r
 \sum_{\emptyset\ne A\subseteq Y}
 \kappa_A(P,Q,t)\,
 \alpha_i^{h-|A|/2}(1-\alpha_i)^{|A|/2}.
\end{equation}
\end{proposition}

\begin{proof}
For each color \(i\), put \(g_i=f_i-\alpha_i1\). Since
\(\widehat f_i(0)=\alpha_i\), we have \(\widehat g_i(0)=0\). Since every
\(\psi_j\) with \(j\ge1\) is orthogonal to \(1\), we also have
\(\widehat g_i(j)=\widehat f_i(j)\) for \(j\ge1\). By Parseval on
\(1^\perp\),
\[
 \sum_{j=1}^{m-1}|\widehat g_i(j)|^2
 =
 \|g_i\|_{2,X}^2.
\]
Because \(f_i\) is an indicator function and has mean \(\alpha_i\),
\[
 \|g_i\|_{2,X}^2
 =
 \frac1m\sum_{C\in X}(f_i(C)-\alpha_i)^2
 =
 \alpha_i(1-\alpha_i).
\]

We now decompose the error term in Theorem~\ref{thm:poset-fourier-expansion}
according to supports. For fixed \(i\) and nonempty \(A\subseteq Y\), the
contribution of all multi-indices with support \(A\) is
\[
 T_{i,A}
 =
 \alpha_i^{h-|A|}
 \sum_{\substack{\mathbf j\in I_m^Y\\ \operatorname{supp}(\mathbf j)=A}}
 \mathcal K_{P,Q,t}(\mathbf j)
 \prod_{D\in A}\widehat g_i(j_D),
\]
where the factor \(\alpha_i^{h-|A|}\) comes from the coordinates
\(D\notin A\), for which \(j_D=0\). By Cauchy--Schwarz and
\eqref{eq:kappa-A-def},
\[
 |T_{i,A}|
 \le
 \alpha_i^{h-|A|}\kappa_A(P,Q,t)
 \left(
 \sum_{(j_D)_{D\in A}\in\{1,\ldots,m-1\}^A}
 \prod_{D\in A}|\widehat g_i(j_D)|^2
 \right)^{1/2}.
\]
The last sum factorizes as
\[
 \prod_{D\in A}\sum_{j=1}^{m-1}|\widehat g_i(j)|^2
 =
 \bigl(\alpha_i(1-\alpha_i)\bigr)^{|A|}.
\]
Therefore
\[
 |T_{i,A}|
 \le
 \kappa_A(P,Q,t)\,
 \alpha_i^{h-|A|/2}(1-\alpha_i)^{|A|/2}.
\]
Summing this estimate over all \(i\in[r]\) and all nonempty \(A\subseteq Y\)
proves \eqref{eq:fourier-parseval-error-bound}.
\end{proof}

For \(1\le k\le h\), define
$$\kappa_k(P,Q,t)=\max\{\kappa_A(P,Q,t):A\subseteq Y,\ |A|=k\}.$$ Let
\(\Delta_{r-1}=\{(\alpha_1,\ldots,\alpha_r)\in[0,1]^r:
\alpha_1+\cdots+\alpha_r=1\}\), and let \([x]_+=\max\{x,0\}\). For
\(\boldsymbol\alpha=(\alpha_1,\ldots,\alpha_r)\in\Delta_{r-1}\), set
$$
 \mathcal F_{P,Q,t,r}(\boldsymbol\alpha)
 =
 \sum_{i=1}^r\alpha_i^h
 -
 \sum_{i=1}^r\sum_{k=1}^h
 \binom hk\kappa_k(P,Q,t)
 \alpha_i^{h-k/2}(1-\alpha_i)^{k/2}.
$$

\begin{corollary}
\label{cor:poset-fourier-multiplicity-bound}
One has
\begin{equation}\label{eq:poset-fourier-multiplicity-bound}
 \operatorname M_{r,t}^{\sharp}(Q\mid P)
 \ge
 \mathcal N^{\sharp}(Q,P)
 \min_{\boldsymbol\alpha\in\Delta_{r-1}}
 \bigl[\mathcal F_{P,Q,t,r}(\boldsymbol\alpha)\bigr]_+.
\end{equation}
\end{corollary}

\begin{proof}
Let \(\chi:\mathcal C_t(Q)\to[r]\) be arbitrary, and let
\(\boldsymbol\alpha(\chi)=(\alpha_1,\ldots,\alpha_r)\) be its color-density
vector. Proposition~\ref{prop:poset-fourier-parseval-bound} gives
\[
 \frac{\sum_{i=1}^rN_i^\sharp(\chi)}{\mathcal N^\sharp(Q,P)}
 \ge
 \sum_{i=1}^r\alpha_i^h
 -
 \sum_{i=1}^r\sum_{\emptyset\ne A\subseteq Y}
 \kappa_A(P,Q,t)
 \alpha_i^{h-|A|/2}(1-\alpha_i)^{|A|/2}.
\]
For every \(A\subseteq Y\) with \(|A|=k\), we have
\(\kappa_A(P,Q,t)\le\kappa_k(P,Q,t)\), and there are \(\binom hk\) such
subsets. Hence the right-hand side is at least
\(\mathcal F_{P,Q,t,r}(\boldsymbol\alpha(\chi))\). The left-hand side is a
normalized count and is therefore nonnegative, so it is at least
\([\mathcal F_{P,Q,t,r}(\boldsymbol\alpha(\chi))]_+\). Since
\(\boldsymbol\alpha(\chi)\in\Delta_{r-1}\), we obtain
\[
 \frac{\sum_{i=1}^rN_i^\sharp(\chi)}{\mathcal N^\sharp(Q,P)}
 \ge
 \min_{\boldsymbol\alpha\in\Delta_{r-1}}
 \bigl[\mathcal F_{P,Q,t,r}(\boldsymbol\alpha)\bigr]_+.
\]
Multiplying by \(\mathcal N^\sharp(Q,P)\) and then taking the minimum over
all colorings \(\chi\) proves \eqref{eq:poset-fourier-multiplicity-bound}.
The remaining assertions follow immediately from the same inequality.
\end{proof}

\section{Concluding Remarks}

In this paper, we introduced a general framework for poset Ramsey multiplicity, covering weak and strong embeddings in nested poset families. For the Boolean lattice $B_n$, we studied the arithmetic Schur-type subfamily
$E_n=\{(S,T,U)\in B_n^3:S\subsetneq T\subsetneq U,\ |S|+|T|=|U|\}$.
We proved that the two-color threshold for this family is $9$, obtained the exact formula
\[
|E_n|=\binom{2n}{n}-[x^n](1+x+x^2)^n-2^n+1,
\]
and showed
\[
2^{\delta n+o(n)}
\leq M^{\mathrm{arith}}_2(B_n)
\leq 2^{\gamma n+o(n)},
\]
where \(\delta\approx 1.356779\) and \(\gamma\approx 1.567837\).
We also established a general lower bound for strong poset multiplicity via double counting, a monotonicity theorem under a uniform containment hypothesis, a first-moment upper bound, and a Lov\'asz Local Lemma criterion for zero strong multiplicity.

We conclude with several natural open problems.

\begin{problem}
Determine the correct exponential rate of \(M^{\mathrm{arith}}_2(B_n)\). More precisely, does the limit
\[
\lim_{n\to\infty}\frac1n\log_2 M^{\mathrm{arith}}_2(B_n)
\]
exist? If it exists, is it equal to one of the entropy exponents above, or does
the true exponent lie strictly between \(\delta\approx 1.356779\) and
\(\gamma\approx 1.567837\)?
\end{problem}

\begin{problem}
For small host posets and small posets $P$, compute the exact value of the Ramsey multiplicity $\operatorname{RM}_{r,t}^\sharp(Q_{n_0}\mid P)$ at the Ramsey threshold $n_0$. Can flag algebra methods be adapted to posets to compute exact values or tighten asymptotic bounds?
\end{problem}

\begin{problem}
For which posets $P$ and host families $\mathcal Q$ is the bound
\[
\operatorname M_{r,t}^\sharp(Q_n\mid P)\leq\frac{\mathcal N^\sharp(Q_n,P)}{r^{\tau_t(P)-1}}
\]
asymptotically tight? Conversely, can this universal upper bound be improved for specific families of posets?
\end{problem}


\begin{thebibliography}{99}

\bibitem{AW17}
M. Axenovich, S. Walzer, Boolean lattices: Ramsey properties and embeddings, \emph{Order} 34 (2017), 287--298.

\bibitem{WinterII}
M. Axenovich, C. Winter, Poset Ramsey number $R(P,Q_n)$. II. $N$-shaped poset, \emph{Order} 41 (2024), 401--418.

\bibitem{AW23}
M. Axenovich, C. Winter, Poset Ramsey numbers: large Boolean lattice versus a fixed poset, \emph{Combin. Probab. Comput.} 32(4) (2023), 638--653.

\bibitem{BP2021}
T. Bohman, F. Peng, A construction for Boolean cube Ramsey numbers, \emph{Order} 40 (2023), 327--333.

\bibitem{BR80}
S.A. Burr, V. Rosta, On the Ramsey multiplicities of graphs--problems and recent results, \emph{J. Graph Theory} 4(4) (1980), 347--361.

\bibitem{CGLMNPV22}
F.-H. Chang, D. Gerbner, W.-T. Li, A. Methuku, D. Nagy, B. Patk\'os, M. Vizer, Rainbow Ramsey problems for the Boolean lattice, \emph{Order} 39 (2022), 453--463.

\bibitem{CCLL20}
H.-B. Chen, Y.-J. Cheng, W.-T. Li, C.-A. Liu, The Boolean rainbow Ramsey number of antichains, Boolean posets, and chains, \emph{Electron. J. Combin.} 27(4) (2020), P4.38.

\bibitem{CFSW22}
D. Conlon, J. Fox, B. Sudakov, F. Wei, Threshold Ramsey multiplicity for odd cycles, \emph{Rev. Un. Mat. Argentina} 64(1) (2022), 49--68.

\bibitem{CFSW23}
D. Conlon, J. Fox, B. Sudakov, F. Wei, Threshold Ramsey multiplicity for paths and even cycles, \emph{Eur. J. Combin.} 107 (2023), 103612.

\bibitem{Chung97}
F.R.K. Chung, \emph{Spectral Graph Theory}, CBMS Regional Conference Series in Mathematics 92, American Mathematical Society, Providence, RI, 1997.

\bibitem{CS18}
C. Cox, D. Stolee, Ramsey numbers for partially-ordered sets, \emph{Order} 35 (2018), 557--579.

\bibitem{D03}
B. Datskovsky, On the number of monochromatic Schur triples, \emph{Adv. Appl. Math.} 31 (2003), 193--198.

\bibitem{DKT91}
D. Duffus, H.A. Kierstead, W.T. Trotter, Fibres and ordered set coloring, \emph{J. Combin. Theory Ser. A} 58(1) (1991), 158--164.


\bibitem{Fox08}
J. Fox, There exist graphs with super-exponential Ramsey multiplicity constant, \emph{J. Graph Theory} 57 (2008), 89--98.

\bibitem{Goodman59}
A.W. Goodman, On sets of acquaintances and strangers at any party, \emph{Amer. Math. Monthly} 66 (1959), 778--783.

\bibitem{GRS90}
R.L. Graham, B.L. Rothschild, J.H. Spencer, \emph{Ramsey Theory}, John Wiley \& Sons, 1990.

\bibitem{GRR96}
R.L. Graham, V. R\"odl, A. Ruci\'nski, On Schur properties of random subsets of integers, \emph{J. Number Theory} 61(2) (1996), 388--408.

\bibitem{GRS99}
D.S. Gunderson, V. R\"odl, A. Sidorenko, Extremal problems for sets forming Boolean algebras and complete partite hypergraphs, \emph{J. Combin. Theory Ser. A} 88(2) (1999), 342--367.

\bibitem{GST84}
J.R. Griggs, J. Stahl, W.T. Trotter Jr., A Sperner theorem on unrelated chains of subsets, \emph{J. Combin. Theory Ser. A} 36 (1984), 124--127.

\bibitem{GMT17}
D. Gr\'osz, A. Methuku, C. Tompkins, An improvement of the general bound on the largest family of subsets avoiding a subposet, \emph{Order} 34(1) (2017), 113--125.

\bibitem{GMT23}
D. Gr\'osz, A. Methuku, C. Tompkins, Ramsey numbers of Boolean lattices, \emph{Bull. Lond. Math. Soc.} 55(2) (2023), 914--932.

\bibitem{Harary82}
F. Harary, Achievement and avoidance games for graphs, \emph{Ann. Discrete Math.} 13 (1982), 111--119.

\bibitem{JLM15}
T. Johnston, L. Lu, K.G. Milans, Boolean algebras and Lubell functions, \emph{J. Combin. Theory Ser. A} 136 (2015), 174--183.

\bibitem{JST96}
C. Jagger, P. \v{S}tov\'i\v{c}ek, A. Thomason, Multiplicities of subgraphs, \emph{Combinatorica} 16 (1996), 123--141.

\bibitem{KM}
D.J. Kleitman, G. Markowsky, On Dedekind's problem: the number of isotone Boolean functions. II, \emph{Trans. Amer. Math. Soc.} 213 (1975), 373--390.

\bibitem{KMOW25}
G. O.H. Katona, Y. Mao, K. Ozeki, Z. Wang, Ramsey numbers for partially-ordered sets, arXiv:2512.14638 [math.CO], 2025.

\bibitem{KMOWY26}
G. O.H. Katona, Y. Mao, K. Ozeki, Z. Wang, G. Yang, Boolean lattice without small rainbow subposets, arXiv:2602.00680 [math.CO], 2026.

\bibitem{LT22}
L. Lu, J.C. Thompson, Poset Ramsey numbers for Boolean lattices, \emph{Order} 39 (2022), 171--185.

\bibitem{McColm91}
G.L. McColm, A ramseyian theorem on products of trees, \emph{J. Combin. Theory Ser. A} 57(1) (1991), 68--75.

\bibitem{MN23}
E. Moss, J.A. Noel, Off-diagonal Ramsey multiplicity, arXiv:2306.17388, 2023.

\bibitem{NR84}
J. Ne\v{s}et\v{r}il, V. R\"odl, Combinatorial partitions of finite posets and lattices--Ramsey lattices, \emph{Algebra Univ.} 19 (1984), 106--119.

\bibitem{ODonnell14}
R. O'Donnell, \emph{Analysis of Boolean Functions}, Cambridge University Press, Cambridge, 2014.

\bibitem{PPSS22}
O. Parczyk, S. Pokutta, C. Spiegel, T. Szab\'o, New Ramsey multiplicity bounds and search heuristics, \emph{Found. Comput. Math.} 25 (2025), 1777--1814.

\bibitem{PRS08}
P. Parrilo, A. Robertson, D. Saracino, On the asymptotic minimum number of monochromatic $3$-term arithmetic progressions, \emph{J. Combin. Theory Ser. A} 115 (2008), 185--192.

\bibitem{Patkos20}
B. Patk\'os, On colorings of the Boolean lattice avoiding a rainbow copy of a poset, \emph{Discrete Appl. Math.} 276 (2020), 108--114.

\bibitem{Patkos26}
B. Patk\'os, 
Anti-Ramsey forbidden poset problems, arXiv:2603.10610v2 [math.CO], 2026. 

\bibitem{RZ98}
A. Robertson, D. Zeilberger, A 2-coloring of $[1,n]$ can have $(1/22)n^2+O(n)$ monochromatic Schur triples, but not less!, \emph{Electron. J. Combin.} 5 (1998), R19.

\bibitem{Rosta04}
V. Rosta, Ramsey theory applications, \emph{Electron. J. Combin.} 11 (2004), Dynamic Survey DS13, 43 pp.

\bibitem{RS23}
J. Ru\'e, C. Spiegel, The Rado multiplicity problem in vector spaces over finite fields, \emph{Finite Fields Appl.} 111 (2026), Paper No. 102782.

\bibitem{S99}
T. Schoen, The number of monochromatic Schur triples, \emph{European J. Combin.} 20 (1999), 855--866.

\bibitem{Sp77}
J. Spencer, Asymptotic lower bounds for Ramsey functions, \emph{Discrete Math.} 20 (1977), 69--76.

\bibitem{StanleyEC1}
R.P. Stanley, \emph{Enumerative Combinatorics, Volume 1}, 2nd ed., Cambridge Studies in Advanced Mathematics 49, Cambridge University Press, Cambridge, 2012.

\bibitem{SW16}
A. Saad, J. Wolf, Ramsey multiplicity of linear patterns in certain finite abelian groups, \emph{Q. J. Math.} 68(1) (2017), 125--140.

\bibitem{TrotterRamsey}
W.T. Trotter, Ramsey theory and partially ordered sets, in: R.L. Graham et al. (eds.), \emph{Contemporary Trends in Discrete Mathematics}, DIMACS Ser. Discrete Math. Theoret. Comput. Sci. 49 (1999), 337--347.

\bibitem{Walzer15}
S. Walzer, \emph{Ramsey variant of the $2$-dimension of posets}, Master Thesis, Karlsruhe Institute of Technology, 2015.

\bibitem{Winter23}
C. Winter, Poset Ramsey number $R(P,Q_n)$. I. Complete multipartite posets, \emph{Order} 41 (2024), 391--399.



\bibitem{WinterIII}
C. Winter, Poset Ramsey number $R(P,Q_n)$. III. Chain compositions and antichains, \emph{Discrete Math.} 347(7) (2024), 114031.

\end{thebibliography}
\end{document}